\documentclass[leqno,12pt]{amsart}
\usepackage{stmaryrd} 
\usepackage{amssymb}
\theoremstyle{plain} 
\newtheorem{theorem}{\indent\sc Theorem}[section]
\newtheorem{lemma}[theorem]{\indent\sc Lemma}
\newtheorem{corollary}[theorem]{\indent\sc Corollary}

\theoremstyle{definition} 
\newtheorem{definition}[theorem]{\indent\sc Definition}
\newtheorem{remark}[theorem]{\indent\sc Remark}

\begin{document}

\title[2D
dissipative quasi-geostrophic flows]{The Temporal decay rate of   solutions to  2D
dissipative quasi-geostrophic flows}

 \author[X. Zhao]{Xiaopeng Zhao$^*$ } 

%

\address{$^*$
School of Science\endgraf
Jiangnan University\endgraf
Wuxi 214122,~~~
P. R. China \endgraf
and
\endgraf
School of Mathematics\endgraf Southeast University\endgraf
Nanjing 210018,~~~P. R. China
}
\email{zhaoxiaopeng@jiangnan.edu.cn}


\begin{abstract}
In this paper, by   using Fourier splitting method and the expanded properties of decay character $r^*$, we establish the algebraic   decay rate of higher order derivative of  solutions to  2D dissipative quasi-geostrophic flows.
\end{abstract}\maketitle

{\small

 \section{Introduction}
Consider the quasi-geostrophic flows in $\mathbb{R}^2$:
\begin{equation} \label{1-1}
\left\{ \begin{aligned}
         &\partial_t\theta+u\cdot\nabla\theta+\nu(-\Delta)^{\alpha}\theta=0,\\
                  &\theta(x,0)=\theta_0(x),
                          \end{aligned} \right.
                          \end{equation}
where $\alpha\in(0,1]$, $\theta(x,t)$ is the potential temperature,  $u=(u_1,u_2)\in\mathbb{R}^2$ is the fluid velocity   determined by the scalar stream function $\Psi$ through
$$
u=(u_1,u_2)=\left(-\frac{\partial}{\partial x_2}\Psi,\frac{\partial}{\partial x_1}\Psi\right)=\nabla^{\perp}\Psi,
$$
and $\nu>0$ is a dissipative coefficient. For convenience, we set $\nu=1$.
The relation between the temperature $\theta$ and the stream function $\Psi$ is
$$
\Lambda\Psi=-\theta,
$$
where $\Lambda$ is the Riesz potential operator defined by the fractional power of $-\Delta$:
$$
\Lambda=(-\Delta)^{\frac12}\quad\hbox{and}\quad \widehat{\Lambda^{2\gamma}f}=\widehat{(-\Delta)^{\gamma}f}=|\xi|^{2\gamma}\hat{f},\quad\forall\gamma>0,
$$
and $\hat{f}$ denotes the Fourier transform of $f$.




Besides
this geophysical connection, system (\ref{1-1}) also shares many features with fundamental fluid motion equations. If the dissipative coefficient $\nu=0$, this system is comparable with the vorticity formulation of the Euler equations \cite{MT}. 
 From the mathematical viewpoint, system (\ref{1-1}) is strikingly similar to the 3D Navier-Stokes equations, although (\ref{1-1}) is considerably simpler than 3D Navier-Stokes equations. According to the selection of $\alpha$, there  are three cases for system (\ref{1-1}):
  subcritical when $\alpha\in(\frac12,1)$, critical when $\alpha=\frac12$ and supercritical when $\alpha\in(0,\frac12)$.
 For the subcritical   and critical ranges   of 2D quasi-geostrophic equations, the smoothness of global solutions has been proved (cf. \cite{Luis,CW,SG,KA} and the reference cited therein). However, it is also an open problem for the supercritical range  of system (\ref{1-1}) (cf. \cite{CM,CD,CA}).

 \begin{remark}
 If $u=\nabla^{\perp}(-\Delta)^{\frac{\beta}2-1}\theta$ with $\beta\in(0,1]$, the system (\ref{1-1}) is called the modified quasi-geostrophic equations, which also studied by many authors (cf. \cite{FNP,CIW,MX,Chae} and reference cited therein).
 \end{remark}

The temporal decay rate of solutions is   an interesting topic in the study of  dissipative equations.  With the aid of the classical Fourier splitting method \cite{S1,S2}, Constantin and Wu \cite{CW}, Schonbek and Schonbek \cite{SS} established some decay results for the subcritical quasi-geostrophic equations. On the basis of the maximal principle, C\'{o}rdoba and C\'{o}rdoba \cite{CD}, Ju \cite{Ju} invistigated the $L^p$-norm decay estimates. Moreover, based on the spectral decomposition of the Laplacian operator and iterative techniques, Dong and Liu \cite{DL} also obtained the upper bounds of weak solution and higher order derivatives of solutions to 2D quasi-geostrophic flows. Kato's method can also be used to study the $L^p$-norm decay rate of solutions to 2D quasi-geostrophic equations \cite{CL,Niche}. We also recall that  Zhou introduced a new method (see~\cite{Zhou2}, some people called Zhou's method) to handle decay rate problems of system (\ref{1-1}).

Recently, in order to characterize the decay rate of dissipative equations more profound, Bjorland and Schonbek\cite{ADE}, Niche and Schonbek \cite{NS} introduced the    idea of decay indicator $P_r $  and decay character $r^* $. Latterly, Brandolese \cite{B} improved the definition of the decay indicator and the decay character by taking advantage of the insight provided by the Littlewood-Paley analysis and the use of Besov spaces. For more details on $P_r$ and $r^*$, we refer to Section 2. We recall that there are some papers studied the decay estimate of solutions to dissipative equations by using the idea of decay character and Fourier splitting method (cf. \cite{A,Niche2,N3,Zhao}) and the reference cited therein).
 It is worth pointing out that, on the basis of Fourier splitting method and decay character $r^*$, Niche and Schonbek \cite{NS},  Ferreira, Niche and Planas \cite{FNP} obtained some new results on the dissipative quasi-geostrophic equations and modified quasi-geostrophic equations.

The following two results on the decay rate of solutions to system (\ref{1-1}) was proved by Niche and Schonbek \cite{NS}.
\begin{lemma}[\cite{NS}]
\label{lem4.2}
Let $\theta_0\in L^2(\mathbb{R}^2)$  have the decay character $r^*=r^*(\theta)\in(-1,\infty)$. Suppose that $\theta(x,t)$ is a weak solution to system (\ref{1-1}). Then
\begin{itemize}
\item  if $r^*\leq1-\alpha$,
$$
\| \theta(t)\|_{L^2}^2\leq C(1+t)^{-\frac1{\alpha}(1+r^*) },\quad\hbox{for~large}~t;
$$

\item  if $r^*\geq1-\alpha$,
$$
\| \theta(t)\|_{L^2}^2\leq C(1+t)^{-\frac1{\alpha}(2-\alpha) },\quad\hbox{for~large}~t.
$$
\end{itemize}
\end{lemma}

\begin{lemma}[\cite{NS}]
\label{lem4.2-1}
Let $\frac12<\alpha<1$, $\alpha\leq s$, $\kappa>0$ and $\theta_0\in H^{\kappa}(\mathbb{R}^2)$  have the decay character $r^*=r^*(\theta)\in(-1,\infty)$. Suppose that $\theta(x,t)$ is a   solution to system (\ref{1-1}). Then
\begin{itemize}
\item  if $r^*\leq1-\alpha$,
$$
\| \Lambda^{\kappa}\theta(t)\|_{L^2}^2\leq C(1+t)^{-\frac1{\alpha}(\kappa+1+r^*) },\quad\hbox{for~large}~t;
$$

\item  if $r^*\geq1-\alpha$,
$$
\|\Lambda^{\kappa} \theta(t)\|_{L^2}^2\leq C(1+t)^{-\frac1{\alpha}(\kappa+
2-\alpha) },\quad\hbox{for~large}~t.
$$
\end{itemize}
\end{lemma}

In this paper, we continue to study the decay characterization of solutions to system (\ref{1-1})  with $\frac12<\alpha<1$. On the basis of the  expanded properties (see Lemma \ref{lem2.3}) of decay character $r^*$   and Fourier splitting method, we establish the temporal decay rate of higher-order derivative of solutions. Our result can be described as follows:
\begin{theorem}
\label{thm1.2}
Let $\frac12<\alpha<1$, $\alpha\leq s$ and $\theta_0\in H^{2\alpha P+  \kappa}(\mathbb{R}^2)$ $(P\in \mathbb{N},~{\kappa}\geq0)$ have the decay character $r^*=r^*(\theta)\in(-1,\infty)$. Suppose that $\theta(x,t)$ is the solution to system (\ref{1-1}). Then
\begin{itemize}
\item  if $r^*\leq1-\alpha$,
\begin{equation}\label{main-1}
\|\partial_t^P\Lambda^{\kappa}\theta(t)\|_{L^2}^2\leq C(1+t)^{-\frac1{\alpha}({\kappa}+1+r^*+2\alpha P) },\quad\hbox{for~large}~t;
\end{equation}

\item  if $r^*\geq1-\alpha$,
\begin{equation}\label{main-2}
\|\partial_t^P\Lambda^{\kappa}\theta(t)\|_{L^2}^2\leq C(1+t)^{-\frac1{\alpha}({\kappa}+2-\alpha+2\alpha P) },\quad\hbox{for~large}~t.
\end{equation}
\end{itemize}

\end{theorem}

The main purpose of this paper is to prove Theorem \ref{thm1.2}, which can be proved by inductive argument:
\begin{itemize}
\item[(1).] Firstly, we prove Theorem \ref{thm1.2} holds for $P=1$ and $\kappa\geq0$;
\item[(2).] Suppose that Theorem \ref{thm1.2} holds for $P\leq p-1$ and $\kappa\geq0$, then we show that it also holds for $P=p$ and $\kappa\geq0$:
\begin{itemize}
\item we show that it holds for $P=p$ and $\kappa=0$;
\item we show that it holds for $P=p$ and $\kappa>0$.
\end{itemize}
\end{itemize}

The rest of this paper is organized as follows. We begin by giving some preliminary results on the properties of  decay character $r^*$ in Section 2. Then, in Section 3, we prove  Theorem \ref{thm1.2}.
\section{Preliminary results and properties of Decay character}
\subsection{Preliminary results}
Denote the Riesz transform in $\mathbb{R}^2$ by $\mathcal{R}_j$, $j=1,2$ as
$$
\widehat{\mathbb{R}_jf}=-i\frac{\xi_j}{|\xi|}\hat{f}(\xi).
$$
The operator $\mathcal{R}^{\perp}$ is defined by
$$
\mathcal{R}^{\perp}f=(-\partial_{x_2}\Lambda^{-1}f, \partial_{x_1}\Lambda^{-1}f)=-(\mathcal{R}_2f,\mathcal{R}_1f),
$$
which implies that the relation between $u$ and $\theta$ is $u=\mathcal{R}^{\perp}\theta$.

\begin{lemma}[\cite{Zhou2}]
\label{Zhou}
There exists a positive constant $C(p)$ depending only on $p$ such that
$$
\|\Lambda^{\delta}u\|_{L^p}\leq C(p)\|\Lambda^{\delta}\theta\|_{L^p},
$$
for all $\delta\geq0$, $1<p<\infty$. If $p=2$, the above inequality is actually an identity.
\end{lemma}

It is easy to see that the following result also holds for system (\ref{1-1}):
\begin{corollary}
\label{Cor}
There exists a positive constant $C(p)$ depending only on $p$  such that
$$
\|\partial_t^{\iota}\Lambda^{\delta}u\|_{L^p}\leq C(p)\|\partial_t^{\iota}\Lambda^{\delta}\theta\|_{L^p},
$$
for all $\iota\in\mathbb{N}$, $\delta\geq0$, $1<p<\infty$.
\end{corollary}

\subsection{Definition and properties of decay character}

 The definitions of decay indicator $P_r(u_0)$ and decay character $r^*$ was  introduced in \cite{ADE,B}.
\begin{definition}[\cite{ADE,B}]
\label{def2.1}
Suppose that $v_0\in L^2(\mathbb{R}^n)$, $B_{\rho}=\{\xi\in\mathbb{R}^n:|\xi|\leq\rho\}$ and $\Lambda=(-\Delta)^{\frac12}$. For $s\geq0$, let the limit
$$
P_r^s(v_0)=\lim_{\rho\rightarrow0}\rho^{-2r-n}\int_{B(\rho}|\xi|^{2s}|\hat{v}_0(\xi)|^2d\xi,
$$
holds for $-\frac n2+s<r<\infty$. Then,   $P_r^s(v_0)$ is called the $s$-decay indicator corresponding to $\Lambda^s v_0$.
\end{definition}
\begin{definition}[\cite{ADE,B}]
\label{def2.2}
The decay character of $\Lambda^sv_0$ denoted by $r_s^*=r_s^*(v_0)$, is the unique $r\in(-\frac n2+s,\infty)$ such that $0<P_r^s(v_0)<\infty$, provided that this number exists. If such $P_r^s(v_0)$ does not exist, we can set $r_s^*=-\frac n2+s$ when $P_r^s(v_0)=\infty$ for all $r\in(-\frac n2+s,\infty)$ or $r_s^*=\infty$ if $P_r^s(v_0)=0$ for all $r\in(-\frac n2+s,\infty)$.
\end{definition}


There's also a lemma on the relation between the decay character of $\Lambda^sv_0$ and that of $v_0$.
\begin{lemma}[\cite{NS}]
Let $v_0\in H^s(\mathbb{R}^n)$, with $s>0$. Then
\begin{itemize}
\item[(1)] if $-\frac n2<r^*(v_0)<\infty$, then $-\frac n2+s<r_s^*(v_0)<\infty$ and $r_s^*(v_0)=s+r^*(v_0)$;
\item[(2)] $r_s^*(v_0)=\infty$ if and only if $r^*(v_0)=\infty$;
\item[(3)] $r^*(v_0)=-\frac n2$ if and only if $r_s^*(v_0)=-\frac n2+s$.
\end{itemize}
\label{HS}
\end{lemma}

\begin{remark}[\cite{A,Z}]The decay character $r^*=r^*(v_0)$  measure the ``order" of $\widehat{v}_0(\xi)$ at $\xi=0$ in frequency space. The theory of \cite{ADE,NS} allows to define the decay character only in the following three situations:
\begin{itemize}
\item[(1)] Either, $\exists r\in(-\frac n2+s,+\infty)$ such that $0<P_r^s(u_0)<+\infty$, and in this case such $r$ is unique,
\item[(2)] Or $\forall r\in(-\frac n2+s,+\infty)$, one has $P_r^s(u_0)=0$,
\item[(3)] Or $\forall r\in(-\frac n2+s,+\infty)$, one has $P_r^s(u_0)=+\infty$.
\end{itemize}
But not in the other cases (e. g. it can   happen that $$
\exists r,r'\in(-\frac n2+s,+\infty)~\hbox{such~that}~P_r^s(u_0)=0~\hbox{and}~P_{r'}^s(u_0)=+\infty.
$$
In addition, it can also happen that the limit defining $P_r^s(u_0)$ does not exist). 
\end{remark}
\subsection{Decay characterization of  the linear system}

Let $0<\alpha\leq1$. Consider the solutions of the linear problem in $ \mathbb{R}^n$ ($n\in\mathbb{N}^+$):
 \begin{equation}\label{2-1}
 \left\{ \begin{aligned}
         & v_t +\Lambda^{2\alpha} v=0, \\
                  &v(x,0)=v_0(x).
                           \end{aligned} \right.
                           \end{equation}
 The solution of system (\ref{2-1}) can be represented by the fundamental solution as
 $$
 v(t)=e^{-t\Lambda^{2\alpha}}v_0=G_{\alpha}*v_0,
 $$
 where $G_{\alpha}$ is given from the Fourier transform as
 $$
 \widehat{G}_{\alpha}(\xi,t)=e^{-|\xi|^{2\alpha}t}.
 $$
   Moreover,
simple calculation shows that
$$
\frac12\frac d{dt}\|v(t)\|_{L^2}^2\leq-C\int_{\mathbb{R}^n}|\xi|^{2\alpha}|\hat{v}|^2d\xi.
$$
The $L^2$-decay characterization of solutions to system (\ref{2-1}) was established by Niche and Schonbek \cite{NS} by using Fourier splitting method.
 \begin{lemma}[\cite{NS}]
 \label{lem2.1}
 Assume that
  $v_0\in L^2(\mathbb{R}^n)$, which has decay character $r^*(v_0)=r^*$, is a solution to system (\ref{2-1}). Then
 \begin{itemize}
 \item[(1).] If $-\frac n2<r^*<+\infty$, there exists two positive constants $C_1,~C_2>0$, such that
 $$
 C_1(1+t)^{-\frac1{\alpha}(\frac n2+r^*)}\leq\|v(t)\|_{L^2}^2 \leq C_2(1+t)^{-\frac1{\alpha}(\frac n2+r^*)};
 $$
 \item[(2).] if $r^*=-\frac n2$, there exists $C=C(\varepsilon)>0$, such that
 $$
 \|v(t)\|_{L^2}^2 \geq C(1+t)^{-\varepsilon},\quad\forall \varepsilon>0,
 $$
 which means the decay of $\|v\|_{L^2}^2 $ is slower than any uniform algebraic rate;
 \item[(3).] if $r^*=+\infty$, there exists a $C>0$ such that
 $$
 \|v(t)\|_{L^2}^2 \leq C(1+t)^{-m},\quad \forall m>0,
 $$
 that is, the decay of $\|v\|_{L^2}^2 $
  is faster than any algebraic rate.
  \end{itemize}
  \end{lemma}
 There's also an lemma on the $\dot{H}^s$-decay rate of solutions to system (\ref{2-1}).
\begin{lemma}[\cite{NS}]
Suppose that   $\Lambda^sv_0\in L^2(\mathbb{R}^n)$ ($s>0$) has decay character $r_s^*=r_s^*(v_0)$. Then
 \begin{itemize}
 \item[(1)] If $-\frac n2\leq r^*<+\infty$, there exists two positive constants $C_1,~C_2>0$, such that
 $$
 C_1(1+t)^{-\frac1{\alpha}(\frac n2+r^*+s)}\leq\|\Lambda^sv(t)\|_{L^2}^2 \leq C_2(1+t)^{-\frac1{\alpha}(\frac n2+r^*+s)};
 $$
 \item[(2)]   if $r^*=+\infty$, there exists a $C>0$ such that
 $$
 \|v(t)\|_{L^2}^2 \leq C(1+t)^{-m},\quad \forall m>0,
 $$
 that is, the decay of $\|\Lambda^sv\|_{L^2}^2 $
  is faster than any algebraic rate.
  \end{itemize}
  \label{lem2.2}
\end{lemma}

We now prove the following result on the decay characterization of solutions to system (\ref{2-1}).
\begin{lemma}
\label{lem2.3}
Let 
 $p\in\mathbb{N}$, $m\in\mathbb{R}^+$ and $v_0\in H^{K}(\mathbb{R}^n)$  has decay character $r^*_m=r^*_m(v_0)$ ($m>0$).   Then, for all $0<  m+2\alpha p\leq K$, the following decay estimates hold:
\begin{itemize}
\item[(1)] If $-\frac n2\leq r^*<\infty$, then there exists  positive constants $C_1$ and $C_2$ such that
$$
C_1(1+t)^{-\frac1{\alpha}(r^*+m+2\alpha p+\frac n2)}\leq \|\partial_t^p\Lambda^mv(t)\|_{L^2}^2 \leq C_2(1+t)^{-\frac1{\alpha}(r^*+m+2\alpha p+\frac n2)};
$$
\item[(2)] if $r^*=\infty$, given any $s>0$, there exists a positive constant $C_2=C_2(s)$ such that
$$
\|\partial_t^p\Lambda^mv(t)\|_{L^2}  \leq C_2(1+t)^{-s},
$$
which means the decay is faster than any algebraic rate.
\end{itemize}
\end{lemma}
\begin{proof}(1).
Applying $\partial_t^p$ to (\ref{2-1}), multiplying by $\partial_t^p\Lambda^{2m}v$, we derive that
\begin{equation}
\label{3-1}
\frac12\frac d{dt}\|\partial_t^p\Lambda^mv\|_{L^2}^2 +\|\partial_t^p\Lambda^{m+\alpha}v\|_{L^2}^2=0.
\end{equation}
Using Plancherel's theorem to (\ref{3-1}), it yields that
\begin{equation}
\label{3-2}
\frac d{dt} \int_{\mathbb{R}^n}  |\xi|^{2m}|\partial_t^p\hat{v}(\xi,t)|^2 d\xi  +2\int_{\mathbb{R}^n} |\xi|^{2m+2\alpha}|\partial_t^p\hat{v}(\xi,t)|^2 d\xi=0.
\end{equation}
Let
$$
B(t)=\left\{\xi\in\mathbb{R}^n:|\xi|\leq\rho(t)=\left(\frac{g'(t)}{2g(t) }\right)^{\frac1{2\alpha}}\right\},
$$
where $g\in C^1(\mathbb{R}^n)$, $g(0)=1$ and $g'(t)>0$, for all $t>0$. Hence,
\begin{equation}
\begin{aligned}\label{3-3}&
\frac d{dt}\left(g(t)\int_{\mathbb{R}^n} |\xi|^{2m}|\partial_t^p\hat{v}(\xi,t)|^2d\xi\right)\\
\leq& g'(t)\int_{B(t)} |\xi|^{2m }|\partial_t^p\hat{v}(\xi,t)|^2d\xi
\\
\leq&g'(t)\int_{B(t)} |\xi|^{2m}\left( |\xi|^{2\alpha} \right)^{2p}e^{-2 |\xi|^{2\alpha} t}| \hat{v}(\xi,t)|^2d\xi
\\
\leq&g'(t)\rho^{2(r^*+m+2\alpha p)+n}\rho^{-2(r^*+m+2\alpha p)-n}\int_{B(t)}|\xi|^{2m+4\alpha p} |\hat{v}_0(\xi)|^2d\xi,
\end{aligned}
\end{equation}
for all $0<m+2\alpha p\leq K$.
Since $P_{r^*_m}(v_0)<\infty$ and $r^*(\Lambda^mv_0)=r^*(v_0)+m$, for all $m\geq0$, there exist $\rho_0>0$ and $C>0$ such that for $0<\rho<\rho_0$,
$$
\rho^{-2(r^*+m)-n}\int_{B(t)}|\xi|^{2m}|\hat{v}_0(\xi)|^2d\xi\leq C,
$$
which imply that \begin{equation}
\begin{aligned}\label{3-4} &
\frac d{dt}\left(g(t)\int_{\mathbb{R}^n} |\xi|^{2m}|\partial_t^p\hat{v}(\xi,t)|^2d\xi\right)
 \leq
Cg'(t)\rho^{2(r^*+m+2\alpha p)+n}  .
\end{aligned}
\end{equation}
Taking $g(t)=(\frac t2+b)^{k}$ with $k>r^*+m+2\alpha p+1+\frac n2$, we have    $\rho(t)=C(b+t)^{-\frac12}$. By (\ref{3-4}), we deduce that
\begin{equation}
\begin{aligned}\nonumber
\|\partial_t^p\Lambda^mv(t)\|_{L^2}^2
\leq& C(b+t)^{-k}+C(b+t)^{-\frac1{\alpha}(r^*+m+2\alpha p+\frac n2)}
\\
\leq&C(b+t)^{-\frac1{\alpha}(r^*+m+2\alpha p+\frac n2)}.
\end{aligned}
\end{equation}
For the lower bound of the decay estimate, as $P_{r_m^*}(v_0)>0$, there exist a $\rho_0>0$ and $C>0$ such that for $\rho\in(0,\rho_0]$,
$$
C<\rho^{-2(r^*+m)-n}\int_{B(t)}|\xi|^{2m}|\hat{v}_0(\xi)|^2d\xi,\quad\forall s\in[0,K].
$$
Therefore,
\begin{equation}\begin{aligned}\nonumber
\|\partial_t^p\Lambda^m v(t)\|_{L^2}^2=&\int_{\mathbb{R}^n}|\xi|^{2m}|\partial_t^p\hat{v}(\xi,t)|^2d\xi
\\
\geq&C\int_{B(t)}|\xi|^{2m+4\alpha p}e^{-2\nu|\xi|^{2\alpha}t}|\hat{v}_0(\xi)|^2d\xi
\\
\geq&C e^{-2\nu\rho^{2\alpha}t}\rho^{2(r^*+m+2\alpha p)+n}\rho^{-2(r^*+m+2\alpha p)-n}\int_{B(t)}|\xi|^{2m+4\alpha p}|\hat{v}_0(\xi)|^2d\xi
\\
\geq& Ce^{-2\nu\rho^{2\alpha}t}\rho^{2(r^*+m+2\alpha p)+n}.
\end{aligned}
\end{equation}
Taking $\rho=\rho_0(1+t)^{-\frac1{2\alpha}}$ yields $
e^{-2\nu\rho^{2\alpha}t}\geq C>0$, so
$$
\|\partial_t^p\Lambda^mv(t)\|_{L^2}^2\geq C_1(1+t)^{-\frac1{\alpha}(r^*+m+2\alpha p+\frac n2)}.
$$

(2). For the case $r^*=\infty$, we have $P_r(v_0)=0$ for any $r\in(-\frac n2,\infty)$. So, for any constant $C=C(r)$, there exists $\rho_0>0$ such that for $0<\rho_0<\rho$,
$$
\rho^{-2(r+m)-n}\int_{B(t)}|\xi|^{2m}|\hat{v}_0(\xi)|^2d\xi<C .
$$
As in the proof of the upper bound in (1) with an inequality similar to (\ref{3-3}), we obtain
$$
\|\partial_t^p\Lambda^mv(t)\|_{L^2}^2 \leq C_2(1+t)^{-s},\quad\forall r\in(-\frac n2,\infty).
$$

\end{proof}

\section{Proof of Theorem \ref{thm1.2}}

The following is a key lemma to prove Theorem \ref{thm1.2}.
\begin{lemma}
\label{lem4.1}Suppose that the assumptions listed in Theorem \ref{thm1.2} are satisfied. Then, for $P>0$, we have
\begin{equation}
\begin{aligned}
|\partial_t^P\widehat{\theta}(\xi,t)|^2\leq& C|\xi|^{4\alpha P}e^{-2|\xi|^{2\nu\alpha}t}|\hat{\theta}_0(\xi )|^2  +C|\xi|^{4\alpha P+2}\left[\int_0^t \|\theta\|_{L^2}^2 ds\right]^2\\&+C\sum_{p=0}^{P-1}\sum_{l=0}^p|\xi|^{4\alpha P-4\alpha p-4\alpha+2} \|\partial_t^l\theta\|_{L^2}^2
\|\partial_t^{p-l}\theta\|_{L^2} ^2.
\label{4-1}\end{aligned}
\end{equation}
\end{lemma}
\begin{proof}
Note that
\begin{equation}
\begin{aligned}
\nonumber
\partial_t^P\hat{\theta}(\xi,t)=&(-1)^P|\xi|^{2\alpha p}e^{-\nu|\xi|^2t}\hat{\theta}_0(\xi)+\sum_{p=0}^{P-1}(- \nu|\xi|^{2\alpha})^{2P-2p-2}
\partial_t^p \widehat{u\cdot\nabla \theta} (\xi,t)\\&+\int_0^t(-\nu)^P|\xi|^{2\alpha P}e^{-\nu|\xi|{2\alpha}(t-s)} \widehat{u\cdot\nabla \theta} (\xi,s)ds,
\end{aligned}
\end{equation}
and
\begin{equation}
\begin{aligned}\nonumber
|\partial_t^p \widehat{u\cdot\nabla \theta} (\xi,t)|\leq&|\partial_t^p\sum_j\xi_j\widehat{u_j\theta}(\xi,t)|
\leq C\sum_{l=0}^p|\xi| \|\partial_t^lu\|_{L^2}
\|\partial_t^{p-l}\theta\|_{L^2} .
\end{aligned}
\end{equation}
Summing up, we complete the proof.
\end{proof}

Now, we establish an auxiliary result.

\begin{lemma}
\label{lem4.3}
Let $\theta_0\in H^2(\mathbb{R}^2)$  have the decay character $r^*=r^*(\theta)\in(-1,\infty)$. Suppose that $\theta(x,t)$ is a weak solution to system (\ref{1-1}). Then
\begin{itemize}
\item[(1).] if $r^*\leq1-\alpha$,
$$
\| \partial_t\theta(t)\|_{L^2}^2\leq C(1+t)^{-\frac1{\alpha}(1+r^*)-2 },\quad\hbox{for~large}~t;
$$

\item[(2).] if $r^*\geq1-\alpha$,
$$
\|\partial_t \theta(t)\|_{L^2}^2\leq C(1+t)^{-\frac1{\alpha}(2-\alpha)-2 },\quad\hbox{for~large}~t.
$$
\end{itemize}
\end{lemma}
\begin{proof}
Applying $\partial_t$ to (\ref{1-1})$_1$, multiplying both side by $\partial_t\theta$, integrating over $\mathbb{R}^2$, we arrive at
\begin{equation}
\begin{aligned}
\nonumber&
\frac12\frac d{dt}\|\partial_t\theta\|_{L^2}^2+ \|\partial_t\Lambda^{\alpha}\theta\|_{L^2}^2
\\
\leq&|\langle\partial_tu\cdot\partial_t\Lambda^{\alpha}\theta,\Lambda^{1-\alpha}\theta\rangle|+|\langle u\cdot\partial_t\Lambda^{\alpha}\theta,\partial_t\Lambda^{1-\alpha}\theta\rangle|\\
&+|\langle\partial_t\Lambda^{1-\alpha}u\cdot\partial_t\Lambda^{\alpha}\theta,\theta\rangle|
+|\langle \Lambda^{1-\alpha}u\cdot\partial_t\Lambda^{\alpha}\theta,\partial_t\Lambda^{1-\alpha}\theta\rangle|
\\
\leq&\|\partial_tu\|_{L^{\frac2{1-\alpha}}}\|\partial_t\Lambda^{\alpha}\theta\|_{L^2}\|\Lambda^{1-\alpha}\theta\|_{L^{\frac2{\alpha}}}
+\|u\|_{L^{\frac2{2\alpha-1}}}\|\partial_t\Lambda^{\alpha}\theta\|_{L^2}\|\partial_t\Lambda^{1-\alpha}\theta\|_{L^{\frac1{1-\alpha}}}
\\
&
+\|\theta\|_{L^{\frac2{2\alpha-1}}}\|\partial_t\Lambda^{\alpha}\theta\|_{L^2}\|\partial_t\Lambda^{1-\alpha}u\|_{L^{\frac1{1-\alpha}}}
+\|\partial_t\theta\|_{L^{\frac2{1-\alpha}}}\|\partial_t\Lambda^{\alpha}\theta\|_{L^2}\|\Lambda^{1-\alpha}u\|_{L^{\frac2{\alpha}}}
\\
\leq&C(\|\Lambda^{2-2\alpha}\theta\|_{L^2}+\|\Lambda^{2-2\alpha}u\|_{L^2})\|\partial_t\Lambda^{\alpha}\theta\|_{L^2}^2
\\
\leq&C(1+t)^{-\min\left\{\frac1{\alpha}( 3-2\alpha+r^*),\frac1{\alpha}(4-3\alpha)\right.\}}\|\partial_t\Lambda^{\alpha}\theta\|_{L^2}^2
\\
\leq&\frac{1}2\|\partial_t\Lambda^{\alpha}\theta\|_{L^2}^2,\quad\hbox{for~large}~t,
\end{aligned}
\end{equation}
that is
\begin{equation}\label{4-2}
\frac d{dt}\|\partial_t\theta\|_{L^2}^2+\|\partial_t\Lambda^{\alpha}\theta\|_{L^2}^2\leq0,\quad\hbox{for~large}~t.
\end{equation}
Applying Plancherel's theorem to (\ref{4-2}), it yields
\begin{equation}
\label{4-3}
\frac d{dt}\int_{\mathbb{R}^2}|\widehat{\partial_t\theta}(\xi,t)|^2d\xi+ \int_{\mathbb{R}^2}|\xi|^{2\alpha}|\widehat{\partial_t\theta}(\xi,t)|^2d\xi\leq0,\quad\hbox{for~large}~t.
\end{equation}
Set
$$
B(t):=\left\{\xi\in\mathbb{R}^2||\xi|^{2\alpha}\leq\frac{g'(t)}{g(t) }\right\},\quad B^c(t):=\mathbb{R}^2\setminus B(t),
$$
where $g(t)$ is a differentiable function of $t$ satisfying
\begin{equation}
\nonumber
g(0)=1,~~g'(t)>0,~~\forall t>0.
\end{equation}
Multiplying (\ref{4-3}) by $g(t)$, we obtain
\begin{equation}
\nonumber
\begin{aligned}
\frac d{dt}\left\{g(t)\int_{\mathbb{R}^3} |\widehat{\partial_t\theta}(\xi,t)|^2 d\xi\right\}
\leq&g'(t)\int_{B(t)} |\widehat{\partial_t\theta}(\xi,t)|^2 d\xi,\quad\hbox{for~large}~t,
\end{aligned}\end{equation}
which together with Lemma \ref{lem4.1} yields that
\begin{equation}\begin{aligned}\label{4-4}
g(t)\int_{\mathbb{R}^2}|\widehat{\partial_t\theta}(\xi,t)|^2d\xi
\leq&C+\int_0^tg'(s)\int_{B(s)}|\xi|^{4\alpha}e^{-2\nu|\xi|^{2\alpha}t}|\hat{\theta}_0(\xi)|^2d\xi ds\\&+C\int_0^tg'(s)\int_{B(s)}|\xi|^{4\alpha+2}\left(\int_0^t\|\theta\|_{L^2}^2ds\right)^2d\xi ds
 .
\end{aligned}
\end{equation}
We estimate now the right hand side of (\ref{4-4}). For the first term, using the estimates from Lemma \ref{lem2.3}, it yields
\begin{equation}
\begin{aligned}\label{4-5}
\int_0^tg'(s)\int_{B(s)}|\xi|^{4\alpha}e^{-2\nu|\xi|^{2\alpha}t}|\hat{\theta}_0(\xi)|^2d\xi ds
\leq  C\int_0^tg'(s)(1+s)^{-\frac1{\alpha}(1+r^*)-2}ds.
\end{aligned}\end{equation}
 For the second term, after integrating in polar coordinates in $B(t)$, by using Lemma \ref{lem4.2}, if $r^*\leq1-\alpha$, we deduce that
\begin{equation}
\begin{aligned}\label{4-6}
&C\int_0^tg'(s)\int_{B(s)}|\xi|^{4\alpha+2}\left(\int_0^t\|\theta\|_{L^2}^2ds\right)^2d\xi ds
\\
\leq&C\int_0^tg'(s)(1+s)^{-2-\frac2{\alpha}}(1+s)^{-\frac2{\alpha}(1+r^*)+2}ds
\leq C\int_0^tg'(s)(1+s)^{-\frac1{\alpha}(4+2r^*)}ds.
\end{aligned}\end{equation}
Thus, for a fixed $r^*$, we choose $g(t)=(1+t)^m$ with some $m>\frac1{\alpha}(1+r^*)+2$. Then $\rho(t)=C(1+t)^{-\frac1{2\alpha}}$.
Combining (\ref{4-4})-(\ref{4-6}) together gives
\begin{equation}\begin{aligned}
\label{4-7} \|\partial_t\theta\|_{L^2}^2
\leq&C(1+t)^{-m}+C(1+t)^{-\frac1{\alpha}(1+r^*)-2}+C(1+t)^{-\frac1{\alpha}(4+2r^*)}
\\
\leq&C(1+t)^{-\frac1{\alpha}(1+r^*)-2}.\quad\hbox{for~large}~t.
\end{aligned}
\end{equation}
On the other hand, since $\alpha\in(\frac12,1)$, if $r^*\geq1-\alpha$, we have
\begin{equation}
\begin{aligned}\label{4-8}
&C\int_0^tg'(s)\int_{B(s)}|\xi|^{4\alpha+2}\left(\int_0^t\|\theta\|_{L^2}^2ds\right)^2d\xi ds
\\
\leq&C\int_0^tg'(s)(1+s)^{-2-\frac2{\alpha}}(1+s)^{ -\frac1{\alpha}(2-\alpha)+2}ds
\\
\leq&C\int_0^tg'(s)(1+s)^{-\frac1{\alpha}(2-\alpha)-\frac2{\alpha}}ds
\leq C\int_0^tg'(s)(1+s)^{-\frac1{\alpha}(2-\alpha)-2}ds.
\end{aligned}\end{equation}
For a fixed $r^*$, we choose $g(t)=(1+t)^m$, for some $m>\frac1{\alpha}(2-\alpha)+2$. Then $\rho(t)=C(1+t)^{-\frac1{2\alpha}}$.
Combining (\ref{4-4}), (\ref{4-5}) and (\ref{4-8}) together gives
\begin{equation}\begin{aligned}
\label{4-9} \|\partial_t\theta\|_{L^2}^2
\leq&C(1+t)^{-m}+C(1+t)^{-\frac1{\alpha}(1+r^*)-2}+C(1+t)^{-\frac1{\alpha}(2-\alpha)-2)}
\\
\leq&C(1+t)^{-\frac1{\alpha}(2-\alpha)-2}.\quad\hbox{for~large}~t.
\end{aligned}
\end{equation}
Hence, we complete the proof.

\end{proof}

The following Lemma \ref{lem4.4} implies that Theorem \ref{thm1.2} holds for $P=1$.

\begin{lemma}
\label{lem4.4}
Let $\theta_0\in H^{2+\kappa}(\mathbb{R}^2)$ ($\kappa\in\mathbb{R}^+$) have the decay character $r^*=r^*(\theta)\in(-1,\infty)$. Suppose that $\theta(x,t)$ is a   solution to system (\ref{1-1}). Then
\begin{itemize}
\item[(1).] if $r^*\leq1-\alpha$,
$$
\| \partial_t\Lambda^{\kappa}\theta(t)\|_{L^2}^2\leq C(1+t)^{-\frac1{\alpha}(1+\kappa+r^*)-2 },\quad\hbox{for~large}~t;
$$

\item[(2).] if $r^*\geq1-\alpha$,
$$
\|\partial_t\Lambda^{\kappa} \theta(t)\|_{L^2}^2\leq C(1+t)^{-\frac1{\alpha}(2+\kappa-\alpha)-2 },\quad\hbox{for~large}~t.
$$
\end{itemize}
\end{lemma}
\begin{proof}
 Applying $\partial_t\Lambda^{\kappa}$ to (\ref{1-1})$_1$, multiplying both side by $\partial_t\Lambda^{\kappa} \theta$, integrating over $\mathbb{R}^2$, we arrive at
 \begin{equation}
 \begin{aligned}&
\frac12\frac d{dt}\|\partial_t\Lambda^{\kappa}\theta\|_{L^2}^2+ \|\partial_t\Lambda^{\kappa+\alpha}\theta\|_{L^2}^2
\\
\leq&C\sum_j|\langle\Lambda^{\kappa-\alpha+1}\partial_t(u_j\theta),\partial_t\Lambda^{\kappa+\alpha}\theta\rangle|
\\
\leq &C\|\partial_t\Lambda^{\kappa+\alpha}\theta\|_{L^2}(\|\Lambda^{\kappa-\alpha+1}(u_j\partial_t\theta)\|_{L^2}+\|\Lambda^{\kappa-\alpha+1}(\partial_tu_j\theta)\|_{L^2})
\\
\leq&C\|\partial_t\Lambda^{\kappa+\alpha}\theta\|_{L^2}\left(\|\Lambda^{\kappa-\alpha+1}u_j\|_{L^4}\|\partial_t\theta\|_{L^4}+\|u_j\|_{L^{\frac2{2\alpha-1}}}\|\partial_t\Lambda^{\kappa-\alpha+1}\theta\|_{L^{\frac1{1-\alpha}}}
\right.\\&
\left.+\|\Lambda^{\kappa-\alpha+1}\theta\|_{L^4}\|\partial_t u_j\|_{L^4}+\|\theta\|_{L^{\frac2{2\alpha-1}}}\|\partial_t\Lambda^{\kappa-\alpha+1}u_j\|_{L^{\frac1{1-\alpha}}}\right)
\\
\leq&C\|\partial_t\Lambda^{\kappa+\alpha}\theta\|_{L^2}\left(\|\Lambda^{\kappa-\alpha+\frac32}u\|_{L^2}\|\partial_t\Lambda^{\kappa+\alpha}
\theta\|_{L^2}^{\frac1{2(\kappa+\alpha)}}\|\partial_t\theta\|_{L^2}^{\frac{2\kappa+2\alpha-1}{2\kappa+2\alpha}}+\|\Lambda^{2(1-\alpha)}u
\|_{L^2}\|\partial_t\Lambda^{\kappa+\alpha}\theta\|_{L^{2}}
\right.\\&
\left.+\|\Lambda^{\kappa-\alpha+\frac32}\theta\|_{L^2}\|\partial_t\Lambda^{\kappa+\alpha}
u\|_{L^2}^{\frac1{2(\kappa+\alpha)}}\|\partial_t u\|_{L^2}^{\frac{2\kappa+2\alpha-1}{2\kappa+2\alpha}}+\|\Lambda^{2(1-\alpha)}\theta
\|_{L^2}\|\partial_t\Lambda^{\kappa+\alpha}u\|_{L^{2}}\right)
\\
\leq&C\|\partial_t\Lambda^{\kappa+\alpha}\theta\|_{L^2}\left(\|\Lambda^{\kappa-\alpha+\frac32}\theta\|_{L^2}\|\partial_t\Lambda^{\kappa+\alpha}
\theta\|_{L^2}^{\frac1{2(\kappa+\alpha)}}\|\partial_t\theta\|_{L^2}^{\frac{2\kappa+2\alpha-1}{2\kappa+2\alpha}}+\|\Lambda^{2(1-\alpha)}\theta
\|_{L^2}\|\partial_t\Lambda^{\kappa+\alpha}\theta\|_{L^{2}}\right)
\\
\leq&\frac{1}2\|\partial_t\Lambda^{\kappa+\alpha}\theta\|_{L^2}^2+C\|\partial_t\theta\|_{L^2}^2\|\Lambda^{\kappa-\alpha+\frac32}\theta\|_{L^2}^{\frac{4\kappa+4\alpha}{2\kappa+2\alpha-1}}
 ,\quad\hbox{for~large}~t,\end{aligned}\nonumber
\end{equation}
which means
\begin{equation}
\label{4-10}
\frac d{dt}\|\partial_t\Lambda^{\kappa}\theta\|_{L^2}^2+ \|\partial_t\Lambda^{\kappa+\alpha}\theta\|_{L^2}^2\leq
C\|\partial_t\theta\|_{L^2}^2\|\Lambda^{\kappa-\alpha+\frac32}\theta\|_{L^2}^{\frac{4\kappa+4\alpha}{2\kappa+2\alpha-1}}
 ,\quad\hbox{for~large}~t.
 \end{equation}
Applying Plancherel's theorem to (\ref{4-10}), it yields
\begin{equation}
\label{4-11}
\begin{aligned}&
\frac d{dt}\int_{\mathbb{R}^2}|\widehat{\partial_t\Lambda^{\kappa}
\theta}(\xi,t)|^2d\xi+ \int_{\mathbb{R}^2}|\xi|^{2\alpha}|\widehat{\partial_t\Lambda^{\kappa}\theta}(\xi,t)|^2d\xi
\\
&\leq C\|\partial_t\theta\|_{L^2}^2\|\Lambda^{\kappa-\alpha+\frac32}\theta\|_{L^2}^{\frac{4\kappa+4\alpha}{2\kappa+2\alpha-1}}
 ,\quad\hbox{for~large}~ t.\end{aligned}
\end{equation}
Multiplying (\ref{4-11}) by $g(t)$, we obtain
\begin{equation}
\nonumber
\begin{aligned}&
\frac d{dt}\left\{g(t)\int_{\mathbb{R}^3} |\widehat{\partial_t\Lambda^{\kappa}\theta}(\xi,t)|^2 d\xi\right\}
\\\leq&g'(t)\int_{B(t)}|\xi|^{2\kappa} |\widehat{\partial_t \theta}(\xi,t)|^2 d\xi+Cg(t)\|\partial_t\theta\|_{L^2}^2\|\Lambda^{\kappa-\alpha+\frac32}\theta\|_{L^2}^{\frac{4\kappa+4\alpha}{2\kappa+2\alpha-1}},\quad\hbox{for~large}~t,
\end{aligned}\end{equation}
which yields that
\begin{equation}\begin{aligned}\label{4-12}
g(t)\int_{\mathbb{R}^2}|\widehat{\partial_t\Lambda^{\kappa}\theta}(\xi,t)|^2d\xi
\leq&C+\int_0^tg'(s)\int_{B(s)}|\xi|^{4\alpha+2\kappa}e^{-2\nu|\xi|^{2\alpha}t}|\hat{\theta}_0(\xi)|^2d\xi ds\\&+C\int_0^tg'(s)\int_{B(s)}|\xi|^{4\alpha+2\kappa+2}\left(\int_0^t\|\theta\|_{L^2}^2ds\right)^2d\xi ds
\\&+C\int_0^tg'(s)s\|\partial_t\theta\|_{L^2}^2\|\Lambda^{\kappa-\alpha+\frac32}\theta\|_{L^2}^{\frac{4\kappa+4\alpha}{2\kappa+2\alpha-1}}ds
,\quad\hbox{for~large}~ t.
\end{aligned}
\end{equation}
The right hand side of (\ref{4-12}) can be estimated one by one. For the first term, using the estimates from Lemma \ref{lem2.3}, we have
\begin{equation}
\begin{aligned}\label{4-13}
\int_0^tg'(s)\int_{B(s)}|\xi|^{4\alpha+2\kappa}e^{-2\nu|\xi|^{2\alpha}t}|\hat{\theta}_0(\xi)|^2d\xi ds
\leq  C\int_0^tg'(s)(1+s)^{-\frac1{\alpha}(1+\kappa+r^*)-2}ds.
\end{aligned}\end{equation}
 For the second   term, after integrating in polar coordinates in $B(t)$, by using Lemma \ref{lem4.2}, if $r^*\leq1-\alpha$, then
\begin{equation}
\begin{aligned}\label{4-14}
&C\int_0^tg'(s)\int_{B(s)}|\xi|^{4\alpha+2\kappa+2}\left(\int_0^t\|\theta\|_{L^2}^2ds\right)^2d\xi ds
\\
\leq&C\int_0^tg'(s)(1+s)^{-2-\frac{\kappa}{\alpha}-\frac2{\alpha}}(1+s)^{-\frac2{\alpha}(1+r^*)+2}ds
\leq C\int_0^tg'(s)(1+s)^{-\frac1{\alpha}(4+2r^*+\kappa)}ds.
\end{aligned}\end{equation}
The third term can be estimate as
\begin{equation}
\begin{aligned}\label{4-15}
&C\int_0^tg'(s)s\|\partial_t\theta\|_{L^2}^2\|\Lambda^{\kappa-\alpha+\frac32}\theta\|_{L^2}^{\frac{4\kappa+4\alpha}{2\kappa+2\alpha-1}}ds
\\
\leq&C\int_0^tg'(s)s(1+s)^{-\frac1{\alpha}(1+r^*)-2}(1+s)^{-\frac1{\alpha}(\kappa-\alpha+\frac52+r^*)\frac{4\kappa+4\alpha}{2\kappa+2\alpha-1}}
\\
\leq&C\int_0^tg'(s)(1+s)^{-\frac1{\alpha}(1+\kappa+r^*)-2}ds,
\end{aligned}\end{equation}
where we have used the fact that $\frac12<\alpha<1$ and $r^*>-1$.
For a fixed $r^*$, we choose $g(t)=(1+t)^m$, for some $m>\frac1{\alpha}(1+\kappa+r^*)+2$. Then $\rho(t)=C(1+t)^{-\frac1{2\alpha}}$.
Combining (\ref{4-12})-(\ref{4-15}) together gives
\begin{equation}\begin{aligned}
\label{4-16} \|\partial_t\Lambda^{\kappa}\theta\|_{L^2}^2
\leq&C(1+t)^{-m}+C(1+t)^{-\frac1{\alpha}(1+\kappa+r^*)-2}+C(1+t)^{-\frac1{\alpha}(4+\kappa+2r^*)}\\
&+C(1+t)^{-\frac1{\alpha}(1+\kappa+r^*)-2}
\\
\leq&C(1+t)^{-\frac1{\alpha}(1+\kappa+r^*)-2}.\quad\hbox{for~large}~t.
\end{aligned}
\end{equation}
On the other hand, since $\alpha\in(\frac12,1)$, if $r^*\geq1-\alpha$, then
\begin{equation}
\begin{aligned}\label{4-17}
&C\int_0^tg'(s)\int_{B(s)}|\xi|^{4\alpha+2\kappa+2}\left(\int_0^t\|\theta\|_{L^2}^2ds\right)^2d\xi ds
\\
\leq&C\int_0^tg'(s)(1+s)^{-2-\frac{\kappa}{\alpha}-\frac2{\alpha}}(1+s)^{ -\frac1{\alpha}(2-\alpha)+2}ds
\\
\leq&C\int_0^tg'(s)(1+s)^{-\frac1{\alpha}(2+\kappa-\alpha)-\frac2{\alpha}}ds
\leq C\int_0^tg'(s)(1+s)^{-\frac1{\alpha}(2+\kappa-\alpha)-2}ds.
\end{aligned}\end{equation}
and
\begin{equation}
\begin{aligned}\label{4-18}
&C\int_0^tg'(s)s\|\partial_t\theta\|_{L^2}^2\|\Lambda^{\kappa-\alpha+\frac32}\theta\|_{L^2}^{\frac{4\kappa+4\alpha}{2\kappa+2\alpha-1}}ds
\\
\leq&C\int_0^tg'(s)s(1+s)^{-\frac1{\alpha}(2-\alpha )-2}(1+s)^{-\frac1{\alpha}(\kappa-2\alpha+\frac72)\frac{4\kappa+4\alpha}{2\kappa+2\alpha-1}}
\\
\leq&C\int_0^tg'(s)(1+s)^{-\frac1{\alpha}(2-\alpha+\kappa)-2}ds,
\end{aligned}\end{equation}
For a fixed $r^*$, we choose $g(t)=(1+t)^m$, for some $m>\frac1{\alpha}(2+\kappa-\alpha)+2$. Then $\rho(t)=C(1+t)^{-\frac1{2\alpha}}$.
Combining (\ref{4-12}), (\ref{4-13}), (\ref{4-17}) and (\ref{4-18}) together gives
\begin{equation}\begin{aligned}
\label{4-19} \|\partial_t\Lambda^{\kappa}\theta\|_{L^2}^2
\leq&C(1+t)^{-m}+C(1+t)^{-\frac1{\alpha}(2+\kappa-\alpha)-2}+C(1+t)^{-\frac1{\alpha}(2+\kappa-\alpha)-2 }
\\
\leq&C(1+t)^{-\frac1{\alpha}(2+\kappa-\alpha)-2}.\quad\hbox{for~large}~t.
\end{aligned}
\end{equation}
Hence, we complete the proof.

\end{proof}

On the basis of Lemma \ref{lem4.3} and Lemma \ref{lem4.4}, we can establish the following result:

\begin{lemma}
\label{lem4.5}
Let $\theta_0\in H^{2p}(\mathbb{R}^2)$ ($p\in\mathbb{N}^+$) have the decay character $r^*=r^*(\theta)\in(-1,\infty)$.  Suppose that $\theta(x,t)$  satisfies (\ref{main-1}) and (\ref{main-2}) for $P\leq p-1$.  Then
\begin{itemize}
\item[(1).] if $r^*\leq1-\alpha$,
$$
\| \partial_t^p\theta(t)\|_{L^2}^2\leq C(1+t)^{-\frac1{\alpha}(1 +r^*)-2p },\quad\hbox{for~large}~t;
$$

\item[(2).] if $r^*\geq1-\alpha$,
$$
\|\partial_t^p \theta(t)\|_{L^2}^2\leq C(1+t)^{-\frac1{\alpha}(2 -\alpha)-2p },\quad\hbox{for~large}~t.
$$
\end{itemize}
\end{lemma}
\begin{proof}
Applying $\partial_t^{p}$ to system (\ref{1-1})$_1$, multiplying both side by $\partial_t^{p}\theta$, integrating over $\mathbb{R}^2$, we arrive at
\begin{equation}
\label{4-20}
 \frac12\frac d{dt}\|\partial_t^p\theta\|_{L^2}^2+ \|\partial_t^p\Lambda^{\alpha}\theta\|_{L^2}^2\leq C|\langle \Lambda^{1-\alpha}\partial_t^p(u_j\theta),\partial_t^p\Lambda^{\alpha}\theta\rangle|.
\end{equation}
Consider the right hand side term of (\ref{4-20}), for large $t$, we have
\begin{equation}
\begin{aligned}
\label{4-21}&
C|\langle \Lambda^{1-\alpha}\partial_t^p(u_j\theta),\partial_t^p\Lambda^{\alpha}\theta\rangle|
\\
=&C|\langle \Lambda^{1-\alpha}\sum_{l=0}^p\partial_t^lu_j\partial_t^{p-l}\theta,\partial_t^p\Lambda^{\alpha}\theta\rangle|
\\
\leq&C|\langle \Lambda^{1-\alpha}  (u_j\partial_t^{p}\theta),\partial_t^p\Lambda^{\alpha}\theta\rangle|+|\langle \Lambda^{1-\alpha}( \partial_t^pu_j \theta),\partial_t^p\Lambda^{\alpha}\theta\rangle|+|\langle \Lambda^{1-\alpha}\sum_{l=0}^p\partial_t^lu_j\partial_t^{p-l}\theta,\partial_t^p\Lambda^{\alpha}\theta\rangle|
\\
\leq&C\|\partial_t^p\Lambda^{\alpha}\theta\|_{L^2}\left(\|\Lambda^{1-\alpha}u\|_{L^{\frac2{\alpha}}}\|\partial_t^p\theta\|_{L^{\frac2{1-\alpha}}}+\|u\|_{L^{\frac2{2\alpha-1}}}\|\partial_t^p\Lambda^{1-\alpha}\theta\|_{L^{\frac1{1-\alpha}}}
\right)\\
&+C\|\partial_t^p\Lambda^{\alpha}\theta\|_{L^2}\left(\|\Lambda^{1-\alpha}\theta\|_{L^{\frac2{\alpha}}}\|\partial_t^pu
\|_{L^{\frac2{1-\alpha}}}+\|\theta\|_{L^{\frac2{2\alpha-1}}}\|\partial_t^p\Lambda^{1-\alpha}u\|_{L^{\frac1{1-\alpha}}}
\right)\\
&+C\|\partial_t^p\Lambda^{\alpha}\theta\|_{L^2}\left(\|\partial_t^lu\|_{L^4}\|\partial_t^{p-l}\Lambda^{1-\alpha}\theta\|_{L^4}+
\|\Lambda^{1-\alpha}\partial_t^lu\|_{L^4}\|\partial_t^{p-l}\theta\|_{L^4}
\right)
\\
\leq&C\|\Lambda^{2-2\alpha}u\|_{L^2}\|\partial_t^p\Lambda^{\alpha}\theta\|_{L^2}^2+C\|\Lambda^{2-2\alpha}\theta\|_{L^2}
\|\partial_t^p\Lambda^{\alpha}\theta\|_{L^2}\|\partial_t^p\Lambda^{\alpha}u\|_{L^2}\\
&+C\|\partial_t^p\Lambda^{\alpha}\theta\|_{L^2}\left(\|\partial_t^l\Lambda^{\frac12}u\|_{L^2}\|\partial_t^{p-l}\Lambda^{\frac32-\alpha}\theta\|_{L^2}+
\|\partial_t^l\Lambda^{\frac32-\alpha}u\|_{L^2}\|\partial_t^{p-l}\Lambda^{\frac12}\theta\|_{L^2}
\right)
\\
\leq&\frac12\|\partial_t^p\Lambda^{\alpha}\theta\|_{L^2}^2+C\left(\|\partial_t^l\Lambda^{\frac12}u\|_{L^2}^2
\|\partial_t^{p-l}\Lambda^{\frac32-\alpha}\theta\|_{L^2}^2+
\|\partial_t^l\Lambda^{\frac32-\alpha}u\|_{L^2}^2\|\partial_t^{p-l}\Lambda^{\frac12}\theta\|_{L^2}^2
\right).
\end{aligned}
\end{equation}
Combining (\ref{4-20}) and (\ref{4-21}) together,  we obtain for large $t$,
\begin{equation}
\label{4-22}
\frac d{dt}\|\partial_t^p\theta\|_{L^2}^2+ \|\partial_t^p\Lambda^{\alpha}\theta\|_{L^2}^2\leq C\left(\|\partial_t^l\Lambda^{\frac12}u\|_{L^2}^2
\|\partial_t^{p-l}\Lambda^{\frac32-\alpha}\theta\|_{L^2}^2+
\|\partial_t^l\Lambda^{\frac32-\alpha}u\|_{L^2}^2\|\partial_t^{p-l}\Lambda^{\frac12}\theta\|_{L^2}^2
\right).
\end{equation}
Applying Plancherel's theorem to (\ref{4-22}), it yields
\begin{equation}
\label{4-23}\begin{aligned}&
\frac d{dt}\int_{\mathbb{R}^2}|\widehat{\partial_t^p\theta}(\xi,t)|^2d\xi+ \int_{\mathbb{R}^2}|\xi|^{2\alpha}|\widehat{\partial_t^p\theta}(\xi,t)|^2d\xi
\\
\leq& C\left(\|\partial_t^l\Lambda^{\frac12}u\|_{L^2}^2
\|\partial_t^{p-l}\Lambda^{\frac32-\alpha}\theta\|_{L^2}^2+
\|\partial_t^l\Lambda^{\frac32-\alpha}u\|_{L^2}^2\|\partial_t^{p-l}\Lambda^{\frac12}\theta\|_{L^2}^2
\right),\quad\hbox{for~large}~t.\end{aligned}
\end{equation}
Multiplying (\ref{4-23}) by $g(t)$, we obtain
\begin{equation}
\nonumber
\begin{aligned} &
\frac d{dt}\left\{g(t)\int_{\mathbb{R}^3} |\widehat{\partial_t^p\theta}(\xi,t)|^2 d\xi\right\}
\\
\leq& g'(t)\int_{B(t)} |\widehat{\partial_t^p\theta}(\xi,t)|^2 d\xi +Cg(t)\left(\|\partial_t^l\Lambda^{\frac12}u\|_{L^2}^2
\|\partial_t^{p-l}\Lambda^{\frac32-\alpha}\theta\|_{L^2}^2\right.
\\&
\left.+
\|\partial_t^l\Lambda^{\frac32-\alpha}u\|_{L^2}^2\|\partial_t^{p-l}\Lambda^{\frac12}\theta\|_{L^2}^2
\right),\quad\hbox{for~large}~t,
\end{aligned}\end{equation}
which yields that
\begin{equation}\begin{aligned}\label{4-24}&
g(t)\int_{\mathbb{R}^2}|\widehat{\partial_t^p\theta}(\xi,t)|^2d\xi
\\
\leq&C \int_0^tg'(s)\int_{B(s)}|\xi|^{ 4\alpha p}e^{-2\nu|\xi|^{2\alpha}t}|\hat{\theta}_0(\xi)|^2d\xi ds\\&+C\int_0^tg'(s)\int_{B(s)}|\xi|^{ 4\alpha p+2}\left(\int_0^t\|\theta\|_{L^2}^2ds\right)^2d\xi ds
\\
&+C\int_0^tg'(s)\int_{B(s)}\sum_{q=0}^{p-1}\sum_{l=0}^q|\xi|^{4\alpha p-4\alpha q-4\alpha+2} \|\partial_t^l\theta\|_{L^2}^2
\|\partial_t^{p-l}\theta\|_{L^2} ^2ds
\\
&+C\int_0^tg'(s)s\left(\|\partial_t^l\Lambda^{\frac12}u\|_{L^2}^2
\|\partial_t^{p-l}\Lambda^{\frac32-\alpha}\theta\|_{L^2}^2+
\|\partial_t^l\Lambda^{\frac32-\alpha}u\|_{L^2}^2\|\partial_t^{p-l}\Lambda^{\frac12}\theta\|_{L^2}^2
\right)ds
 .
\end{aligned}
\end{equation}
The first term of the right hand side of (\ref{4-24}) can be estimated by using the estimates from Lemma \ref{lem2.3}:
\begin{equation}
\begin{aligned}\label{4-25}
\int_0^tg'(s)\int_{B(s)}|\xi|^{ 4\alpha p}e^{-2\nu|\xi|^{2\alpha}t}|\hat{\theta}_0(\xi)|^2d\xi ds
\leq  C\int_0^tg'(s)(1+s)^{-\frac1{\alpha}(1+r^*)-2p}ds.
\end{aligned}\end{equation}
 For the second term  of the right hand side of (\ref{4-24}), after integrating in polar coordinates in $B(t)$, by using Lemma \ref{lem4.2}, if $r^*\leq1-\alpha$, then
\begin{equation}
\begin{aligned}\label{4-26}
&C\int_0^tg'(s)\int_{B(s)}|\xi|^{ 4\alpha p+2}\left(\int_0^t\|\theta\|_{L^2}^2ds\right)^2d\xi ds
\\
\leq&C\int_0^tg'(s)(1+s)^{-\frac2{\alpha}-2  p}(1+s)^{-\frac2{\alpha}(1+r^*)+2}ds
\leq C\int_0^tg'(s)(1+s)^{-\frac1{\alpha}(4+2r^*)-2p+2}ds.
\end{aligned}\end{equation}
If $r^*\leq1-\alpha$, the third term satisfies
\begin{equation}\begin{aligned}
\label{4-26-1}&
C\int_0^tg'(s)\int_{B(s)}\sum_{q=0}^{p-1}\sum_{l=0}^q|\xi|^{4\alpha p-4\alpha q-4\alpha+2} \|\partial_t^l\theta\|_{L^2}^2
\|\partial_t^{p-l}\theta\|_{L^2} ^2ds
\\
\leq&C\int_0^tg'(s)(1+s)^{-2p+2q+2-\frac2{\alpha}}(1+s)^{-\frac1{\alpha}(1+r^*)-2l}(1+s)^{-\frac1{\alpha}(1+r^*)-2(q-l)}ds
\\
\leq&C\int_0^tg'(s)(1+s)^{-\frac2{\alpha}(1+ r^*)-2p+2-\frac2{\alpha}}ds,\quad\hbox{for~large}~t.
\end{aligned}\end{equation}
For the fourth term of right hand side  of (\ref{4-24}), if $r^*\leq1-\alpha$, then
\begin{equation}
\begin{aligned}\label{4-27}
&C\int_0^tg'(s)s\left(\|\partial_t^l\Lambda^{\frac12}u\|_{L^2}^2
\|\partial_t^{p-l}\Lambda^{\frac32-\alpha}\theta\|_{L^2}^2+
\|\partial_t^l\Lambda^{\frac32-\alpha}u\|_{L^2}^2\|\partial_t^{p-l}\Lambda^{\frac12}\theta\|_{L^2}^2
\right)ds
\\
\leq&C\int_0^tg'(s)s(1+s)^{-\frac1{\alpha}(\frac32+r^*)-\frac1{\alpha}(\frac52-\alpha+r^*)-2p}
\\
\leq&C\int_0^tg'(s)(1+s)^{-\frac1{\alpha}(4+2r^*)-2p+2}ds,\quad\hbox{for~large}~t.
\end{aligned}\end{equation}
For a fixed $r^*$, we choose $g(t)=(1+t)^m$, for some $m>\frac1{\alpha}(4+2r^*)-2p+2$. Then $\rho(t)=C(1+t)^{-\frac1{2\alpha}}$.
Combining (\ref{4-24})-(\ref{4-27}) together gives
\begin{equation}\begin{aligned}
\label{4-28} \|\partial_t^p\theta\|_{L^2}^2
\leq&C(1+t)^{-m}+C(1+t)^{-\frac1{\alpha}(1+r^*)-2p}+C(1+t)^{-\frac1{\alpha}(4+2r^*)-2p+2}\\&+C(1+t)^{-\frac2{\alpha}(1+ r^*)-2p+2-\frac2{\alpha}}
\\
\leq&C(1+t)^{-\frac1{\alpha}(1+r^*)-2p}.\quad\hbox{for~large}~t.
\end{aligned}
\end{equation}
On the other hand, since $\alpha\in(\frac12,1)$, if $r^*\geq1-\alpha$, by using the same method as before, for large $t$, we easily obtain
\begin{equation}
\begin{aligned}\label{4-29}
 C\int_0^tg'(s)\int_{B(s)}|\xi|^{4\alpha+4\alpha p+2}\left(\int_0^t\|\theta\|_{L^2}^2ds\right)^2d\xi ds
\leq&
 C\int_0^tg'(s)(1+s)^{-\frac1{\alpha}(2-\alpha)-2p}ds.
\end{aligned}\end{equation}
\begin{equation}\begin{aligned}
\label{4-29-1}&
C\int_0^tg'(s)\int_{B(s)}\sum_{q=0}^{p-1}\sum_{l=0}^q|\xi|^{4\alpha p-4\alpha q-4\alpha+2} \|\partial_t^l\theta\|_{L^2}^2
\|\partial_t^{p-l}\theta\|_{L^2} ^2ds
\\
\leq&C\int_0^tg'(s)(1+s)^{-\frac2{\alpha}(2-\alpha)-2p+2-\frac2{\alpha}}ds.
\end{aligned}\end{equation}
and
\begin{equation}
\begin{aligned}\label{4-30}
&C\int_0^tg'(s)s\left(\|\partial_t^l\Lambda^{\frac12}u\|_{L^2}^2
\|\partial_t^{p-l}\Lambda^{\frac32-\alpha}\theta\|_{L^2}^2+
\|\partial_t^l\Lambda^{\frac32-\alpha}u\|_{L^2}^2\|\partial_t^{p-l}\Lambda^{\frac12}\theta\|_{L^2}^2
\right)ds
\\
\leq&C\int_0^tg'(s)(1+s)^{-\frac1{\alpha}(6-2\alpha)-2p+2}ds,\quad\hbox{for~large}~t.
\end{aligned}\end{equation}
Hence,
For a fixed $r^*$, we choose $g(t)=(1+t)^m$, for some $m>\frac1{\alpha}(2-\alpha)+2p$. Then $\rho(t)=C(1+t)^{-\frac1{2\alpha}}$.
Combining (\ref{4-24}), (\ref{4-25}), (\ref{4-29})-(\ref{4-30}) together gives
\begin{equation}\begin{aligned}
\label{4-31} \|\partial_t^p\theta\|_{L^2}^2
\leq C(1+t)^{-\frac1{\alpha}(2-\alpha)-2p}.\quad\hbox{for~large}~t.
\end{aligned}
\end{equation}
The proof is complete.

\end{proof}

\begin{lemma}
\label{lem4.6}
Let $\theta_0\in H^{2p+\kappa}(\mathbb{R}^2)$ ($\kappa\in\mathbb{R}^+,~p\in\mathbb{N}^+$) have the decay character $r^*=r^*(\theta)\in(-1,\infty)$. Suppose that $\theta(x,t)$  satisfies (\ref{main-1}) and (\ref{main-2}) for $P\leq p-1$. Then
\begin{itemize}
\item[(1).] if $r^*\leq1-\alpha$,
$$
\| \partial_t^p\Lambda^{\kappa}\theta(t)\|_{L^2}^2\leq C(1+t)^{-\frac1{\alpha}(1+\kappa+r^*)-2p },\quad\hbox{for~large}~t;
$$

\item[(2).] if $r^*\geq1-\alpha$,
$$
\|\partial_t^p\Lambda^{\kappa} \theta(t)\|_{L^2}^2\leq C(1+t)^{-\frac1{\alpha}(2+\kappa-\alpha)-2p },\quad\hbox{for~large}~t.
$$
\end{itemize}
\end{lemma}
\begin{proof}First of all, we consider the case $r^*\leq1-\alpha$.
 Applying $\partial_t^p\Lambda^{\kappa}$ to (\ref{1-1})$_1$, multiplying both side by $\partial_t^p\Lambda^{\kappa} \theta$, integrating over $\mathbb{R}^2$, we arrive at
 \begin{equation}
 \begin{aligned}&
\frac12\frac d{dt}\|\partial_t^p\Lambda^{\kappa}\theta\|_{L^2}^2+ \|\partial_t^p\Lambda^{\kappa+\alpha}\theta\|_{L^2}^2
\\
\leq &C\|\partial_t^p\Lambda^{\kappa+\alpha}\theta\|_{L^2} \|\Lambda^{\kappa-\alpha+1}(u_j\partial_t^p
\theta)\|_{L^2}+C\|\partial_t^p\Lambda^{\kappa+\alpha}\theta\|_{L^2}\|\Lambda^{\kappa-\alpha+1}(\partial_t^pu_j\theta)\|_{L^2}
\\&+C\|\partial_t^p\Lambda^{\kappa+\alpha}\theta\|_{L^2}\sum_{l=1}^{p-1}\|\Lambda^{\kappa-\alpha+1}(\partial_t^l
u_j\partial_t^{p-l}\theta)\|_{L^2}
\\
=&J_1+J_2+J_3 .\end{aligned}\label{4-32}
\end{equation}
We will estimate $J_1$, $J_2$ and $J_3$ one by one. Firstly,
\begin{equation}
\begin{aligned}
\label{4-33}
J_1\leq&C\|\partial_t^p\Lambda^{\kappa+\alpha}\theta\|_{L^2}\left(\|\Lambda^{\kappa-\alpha+1}u\|_{L^{4}}\|\partial_t^p\theta\|_{L^4}+\|u\|_{L^{\frac2{2\alpha-1}}}\|\partial_t^p
\Lambda^{\kappa-\alpha+1}\theta\|_{L^{\frac1{1-\alpha}}} \right)
\\
\leq&C\|\partial_t^p\Lambda^{\kappa+\alpha}\theta\|_{L^2}\left(\|\Lambda^{\kappa-\alpha+\frac32}u
\|_{L^{2}} \|\partial_t^p\Lambda^{\kappa+\alpha}\theta\|_{L^2}^{\frac1{2\kappa+2\alpha}}\|\partial_t^p\theta\|_{L^2}^{\frac{2\kappa+2\alpha-1}{2\kappa+2\alpha}}\right.
\\
&\left.+\|\Lambda^{2-2\alpha}u\|_{L^{2}}\|\partial_t^p
\Lambda^{\kappa+\alpha}\theta\|_{L^2} \right)
\\
\leq&\frac{1}{12}\|\partial_t^p\Lambda^{\kappa+\alpha}\theta\|_{L^2}^2+C\|\partial_t^p\theta\|_{L^2}^2\|\Lambda^{\kappa-\alpha+\frac32}\theta\|_{L^2}^{\frac{4\kappa+4\alpha}{2\kappa+2\alpha-1}}
+C(1+t)^{-  \frac1{\alpha}( 3-2\alpha+r^*) }\|\partial_t^p\Lambda^{\kappa+\alpha}\theta\|_{L^2}^2
\\
\leq&\frac{1}6\|\partial_t^p\Lambda^{\kappa+\alpha}\theta\|_{L^2}^2+C\|\partial_t^p\theta\|_{L^2}^2\|\Lambda^{\kappa-\alpha+\frac32}\theta\|_{L^2}^{\frac{4\kappa+4\alpha}{2\kappa+2\alpha-1}}
\\
\leq&\frac{1}6\|\partial_t^p\Lambda^{\kappa+\alpha}\theta\|_{L^2}^2+C(1+t)^{-\frac1{\alpha}(1+r^*)-2p}(1+t)^{-\frac1{\alpha}(\kappa-\alpha+\frac52+r^*)\frac{4\kappa+4\alpha}{2\kappa+2\alpha-1}}\\
\leq&\frac{1}6\|\partial_t^p\Lambda^{\kappa+\alpha}\theta\|_{L^2}^2+C(1+t)^{-\frac1{\alpha}(1+\kappa+r^*)-2p-1},\quad\hbox{for~large}~t.
\end{aligned}
\end{equation}
Similarly,
\begin{equation}
\label{4-34}\begin{aligned}
J_2
\leq &\frac{1}6\|\partial_t^p\Lambda^{\kappa+\alpha}\theta\|_{L^2}^2+C\|\partial_t^p\theta\|_{L^2}^2\|\Lambda^{\kappa-\alpha+\frac32}\theta\|_{L^2}^{\frac{4\kappa+4\alpha}{2\kappa+2\alpha-1}}
\\
\leq&\frac{1}6\|\partial_t^p\Lambda^{\kappa+\alpha}\theta\|_{L^2}^2+C(1+t)^{-\frac1{\alpha}(1+\kappa+r^*)-2p-1},\quad\hbox{for~large}~t.
\end{aligned} \end{equation}
On the other hand, we have
\begin{equation}\begin{aligned}
\label{4-35}
J_3\leq&C\|\partial_t^p\Lambda^{\kappa+\alpha}\theta\|_{L^2}\sum_{l=1}^{p-1}\left(\|\Lambda^{\kappa-\alpha+1}\partial_t^lu\|_{L^4}\|\partial_t^{p-l}\theta\|_{L^4}
+\|\partial_t^lu\|_{L^4}\|\Lambda^{\kappa-\alpha+1}\theta\|_{L^4}
\right)\\
\leq&C\|\partial_t^p\Lambda^{\kappa+\alpha}\theta\|_{L^2}\sum_{l=1}^{p-1}\left(\|\Lambda^{\kappa-\alpha+\frac32}\partial_t^lu\|_{L^2}\|\partial_t^{p-l}\Lambda^{\frac12}
\theta\|_{L^2}
+\|\partial_t^l\Lambda^{\frac12}u\|_{L^2}\|\Lambda^{\kappa-\alpha+\frac32}\theta\|_{L^2}
\right)
\\
\leq&\frac{1}6\|\partial_t^p\Lambda^{\kappa+\alpha}\theta\|_{L^2}^2+C\sum_{l=1}^{p-1}\left(\|\Lambda^{\kappa-\alpha+\frac32}\partial_t^l\theta\|_{L^2}^2
\|\partial_t^{p-l}\Lambda^{\frac12}
\theta\|_{L^2}^2
+\|\partial_t^l\Lambda^{\frac12}\theta\|_{L^2}^2\|\Lambda^{\kappa-\alpha+\frac32}\theta\|_{L^2}^2
\right)\\
\leq&\frac{1}6\|\partial_t^p\Lambda^{\kappa+\alpha}\theta\|_{L^2}^2+C(1+t)^{-\frac1{\alpha}(\kappa-\alpha+\frac52+r^*)-2l}
(1+t)^{-\frac1{\alpha}(\frac32+r^*)-2(p-l)}\\
\leq&\frac{1}6\|\partial_t^p\Lambda^{\kappa+\alpha}\theta\|_{L^2}^2+C(1+t)^{-\frac1{\alpha}(\kappa+4+2r^*)-2p+1},\quad\hbox{for~large}~t.
\end{aligned} \end{equation}
Owning to (\ref{4-32})-(\ref{4-35}), we derive that
\begin{equation}\begin{aligned}
\label{4-36}&\frac d{dt}\|\partial_t^p\Lambda^{\kappa}\theta\|_{L^2}^2+ \|\partial_t^p\Lambda^{\kappa+\alpha}\theta\|_{L^2}^2\\
\leq&  C(1+t)^{-\frac1{\alpha}(1+\kappa+r^*)-2p-1}+C(1+t)^{-\frac1{\alpha}(\kappa+4+2r^*)-2p+1},\quad\hbox{for~large}~t.
\end{aligned} \end{equation}
Applying Plancherel's theorem to (\ref{4-36}), it yields
\begin{equation}
\label{4-37}\begin{aligned}&
\frac d{dt}\int_{\mathbb{R}^2}|\widehat{\partial_t^p\Lambda^{\kappa}\theta}(\xi,t)|^2d\xi+ \int_{\mathbb{R}^2}|\xi|^{2\alpha}|\widehat{\partial_t^p\Lambda^{\kappa}\theta}(\xi,t)|^2d\xi
\\
&\leq C(1+t)^{-\frac1{\alpha}(1+\kappa+r^*)-2p-1}+C(1+t)^{-\frac1{\alpha}(\kappa+4+2r^*)-2p+1},\quad\hbox{for~large}~t.\end{aligned}
\end{equation}
Multiplying (\ref{4-37}) by $g(t)$, we obtain
\begin{equation}
\nonumber
\begin{aligned} &
\frac d{dt}\left\{g(t)\int_{\mathbb{R}^3} |\widehat{\partial_t^p\Lambda^{\kappa}\theta}(\xi,t)|^2 d\xi\right\}
\\
\leq& g'(t)\int_{B(t)} |\widehat{\partial_t^p\Lambda^{\kappa}\theta}(\xi,t)|^2 d\xi+Cg(t)(1+t)^{-\frac1{\alpha}(1+\kappa+r^*)-2p-1} \\&+Cg(t)(1+t)^{-\frac1{\alpha}(\kappa+4+2r^*)-2p+1},\quad\hbox{for~large}~t,
\end{aligned}\end{equation}
which yields that
\begin{equation}\begin{aligned}\label{4-38}&
g(t)\int_{\mathbb{R}^2}|\widehat{\partial_t^p\Lambda^{\kappa}\theta}(\xi,t)|^2d\xi
\\\leq&C+\int_0^tg'(s)\int_{B(s)}|\xi|^{2\kappa+4\alpha+4\alpha p}e^{-2\nu|\xi|^{2\alpha}t}|\hat{\theta}_0(\xi)|^2d\xi ds\\&+C\int_0^tg'(s)\int_{B(s)}|\xi|^{2\kappa+4\alpha+4\alpha p+2}\left(\int_0^t\|\theta\|_{L^2}^2ds\right)^2d\xi ds
\\
&+C\int_0^tg'(s)\int_{B(s)}\sum_{q=0}^{p-1}\sum_{l=0}^q|\xi|^{2\kappa+4\alpha p-4\alpha q-4\alpha+2} \|\partial_t^l\theta\|_{L^2}^2
\|\partial_t^{p-l}\theta\|_{L^2} ^2ds
\\
&+C\int_0^tg'(s)(1+t)^{-\frac1{\alpha}(1+\kappa+r^*)-2p} +C\int_0^tg'(s)s(1+s)^{-\frac1{\alpha}(\kappa+4+2r^*)-2p+1}ds
 .
\end{aligned}
\end{equation}
We estimate now the right hand side of (\ref{4-38}). For the first term, using the estimates from Lemma \ref{lem2.3}, we have
\begin{equation}
\begin{aligned}\label{4-39}
\int_0^tg'(s)\int_{B(s)}|\xi|^{2\kappa+4\alpha+4\alpha p}e^{-2\nu|\xi|^{2\alpha}t}|\hat{\theta}_0(\xi)|^2d\xi ds
\leq  C\int_0^tg'(s)(1+s)^{-\frac1{\alpha}(\kappa+1+r^*)-2p}ds.
\end{aligned}\end{equation}
 For the second term, after integrating in polar coordinates in $B(t)$, by using Lemma \ref{lem4.2},  we get
\begin{equation}
\begin{aligned}\label{4-40}
&C\int_0^tg'(s)\int_{B(s)}|\xi|^{2\kappa+4\alpha+4\alpha p+2}\left(\int_0^t\|\theta\|_{L^2}^2ds\right)^2d\xi ds
\\
\leq&C\int_0^tg'(s)(1+s)^{-2-\frac2{\alpha}-2\alpha p}(1+s)^{-\frac2{\alpha}(\kappa+1+r^*)+2}ds
\\
\leq&C\int_0^tg'(s)(1+s)^{-\frac1{\alpha}(\kappa+4+2r^*)-2p}ds.
\end{aligned}\end{equation}
For the third term of right hand side  of (\ref{4-38}),
\begin{equation}\begin{aligned}
\label{4-26-2}&
C\int_0^tg'(s)\int_{B(s)}\sum_{q=0}^{p-1}\sum_{l=0}^q|\xi|^{2\kappa+4\alpha p-4\alpha q-4\alpha+2} \|\partial_t^l\theta\|_{L^2}^2
\|\partial_t^{p-l}\theta\|_{L^2} ^2ds
\\
\leq&C\int_0^tg'(s)(1+s)^{-\frac{\kappa}{\alpha}-2p+2q+2-\frac2{\alpha}}(1+s)^{-\frac1{\alpha}(1+r^*)-2l}(1+s)^{-\frac1{\alpha}(1+r^*)-2(q-l)}ds
\\
\leq&C\int_0^tg'(s)(1+s)^{-\frac{\kappa}{\alpha}-\frac2{\alpha}(1+ r^*)-2p+2-\frac2{\alpha}}ds.
\end{aligned}\end{equation}
We estimate the last term as follows:
\begin{equation}
\begin{aligned}\label{4-41}
  \int_0^tg'(s)s(1+s)^{-\frac1{\alpha}(\kappa+4+2r^*)-2p+1}ds\leq  \int_0^tg'(s) (1+s)^{-\frac1{\alpha}(\kappa+4+2r^*)-2p+2}ds,~\hbox{for~large}~t.
\end{aligned}\end{equation}
For a fixed $r^*$, we choose $g(t)=(1+t)^m$, for some $m>\frac1{\alpha}(\kappa+4+2r^*)-2p+2$. Then $\rho(t)=C(1+t)^{-\frac1{2\alpha}}$.
Combining (\ref{4-24})-(\ref{4-27}) together gives
\begin{equation}\begin{aligned}
\label{4-42} \|\partial_t^p\Lambda^{\kappa}\theta\|_{L^2}^2
\leq&C(1+t)^{-m}+C(1+t)^{-\frac1{\alpha}(1+\kappa+r^*)-2p}+C(1+t)^{-\frac1{\alpha}(\kappa+4+2r^*)-2p+2}\\
&+C(1+t)^{-\frac{\kappa}{\alpha}-\frac2{\alpha}(1+ r^*)-2p+2-\frac2{\alpha}}
\\
\leq&C(1+t)^{-\frac1{\alpha}(1+\kappa+r^*)-2p}.\quad\hbox{for~large}~t.
\end{aligned}
\end{equation}
For the case $r^*\geq1-\alpha$, by using the same method as above, we easily obtain
\begin{equation}\begin{aligned}
\label{4-43} \|\partial_t^p\Lambda^{\kappa}\theta\|_{L^2}^2
\leq&C(1+t)^{-\frac1{\alpha}(\kappa+2-\alpha)-2p},\quad\hbox{for~large}~t.
\end{aligned}
\end{equation}
Owning to (\ref{4-2}) and (\ref{4-43}), we complete the proof.

\end{proof}

Now, we give the proof of Theorem \ref{thm1.2}.
\begin{proof}[Proof of Theorem \ref{thm1.2}]
  Lemmas \ref{lem4.3} and \ref{lem4.4} imply  that Theorem \ref{thm1.2} holds for $p=1$;
 Suppose that Theorem \ref{thm1.2} holds for $p\leq P-1$, then Lemmas \ref{lem4.5}-\ref{lem4.6} show that it also holds for $p=P$. Hence, we complete the proof of Theorem \ref{thm1.2}.

\end{proof}

\section*{Acknowledgement} XZ was supported by   NNSFC (grant No. 11401258)  and China Postdoctoral Science Foundation (grant No. 2015M581689
). 
 This work was done when XZ was visiting the Institute of Mathematics for Industry of Kyushu University. He appreciate the hospitality of Prof. Fukumoto, MS. Sasaguri and IMI.

}

\end{document}